\documentclass[11pt,reqno,fleqn]{article}
\usepackage{amsfonts,amssymb,mathrsfs,amsmath,amsthm,relsize}
\usepackage[top=20mm,bottom=20mm,left=20mm,right=20mm]{geometry}
\usepackage[T1]{fontenc}
\usepackage{times}
\usepackage{multicol}
\usepackage{graphicx}
\usepackage{graphics}
\usepackage{enumitem}
\usepackage{cite}
\usepackage{mathtools}
\usepackage{etoolbox}
\usepackage{fancyhdr}
\fancypagestyle{titlepage}
{
   \fancyhf{}
   \lhead{The publication is available at https://www.tandfonline.com/doi/full/10.1080/16583655.2022.2071046}
   \rhead{}
   \fancyfoot[C]{\thepage}
}
\allowdisplaybreaks
\makeatletter
\let\@fnsymbol\@arabic
\makeatother
\newcommand\blfootnote[1]{%
  \begingroup
  \renewcommand\thefootnote{}\footnote{#1}%
  \addtocounter{footnote}{-1}%
  \endgroup
}
\theoremstyle{definition}
\newtheorem{theorem}{Theorem}[section]
\newtheorem{definition}[theorem]{Definition}
\newtheorem{lemma}[theorem]{Lemma}
\newtheorem{corollary}[theorem]{Corollary} 
\newtheorem{example}[theorem]{Example}
\newtheorem{remark}[theorem]{Remark}

\allowdisplaybreaks
\makeatletter
\patchcmd{\@maketitle}{\begin{center}}{\begin{flushleft}}{}{}
\patchcmd{\@maketitle}{\begin{tabular}[t]{c}}{\begin{tabular}[t]{@{}l}}{}{}
\patchcmd{\@maketitle}{\end{center}}{\end{flushleft}}{}{}
\makeatother
\begin{document}
\title{On the convergence of sequences in $\mathbb{R}^+$ through weighted geometric means via multiplicative calculus and application to intuitionistic fuzzy numbers}
\author{Enes Yavuz}
\date{{\small Department of Mathematics, Manisa Celal Bayar University, Manisa, Turkey\\E-mail: enes.yavuz@cbu.edu.tr}}
\maketitle
\thispagestyle{titlepage}
\blfootnote{\rule{0.63cm}{0cm}\emph{\!\!Mathematics Subject Classification:} 03E72, 40A05, 40E05}

\noindent\textbf{Abstract:} We define weighted geometric mean method of convergence for sequences in $\mathbb{R}^+$ by using multiplicative calculus and obtain necessary and sufficient conditions under which convergence of sequences in $\mathbb{R}^+$ follows from convergence of their weighted geometric means. We also obtain multiplicative analogues of Schmidt type slow oscillation condition and Landau type two-sided condition for the convergence in particular. Besides, we introduce the concepts of  $^{\oplus}$convergence, $^{\otimes}$convergence, $(\bar{N},p)-^{\oplus}$convergence, $(\bar{G},p)-^{\otimes}$convergence for sequences of intuitionistic fuzzy numbers (IFNs) and apply the aforementioned conditions to achieve convergence in intuitionistic fuzzy number space. Examples of sequences are also given to illustrate the proposed methods of convergence.
\section{Introduction}
Multiplicative calculus\cite{katz,stanley} is alternative to classical calculus and uses ratios instead of differences in order to measure deviations and to compare numbers. The operations multiplication and division are crucial in multiplicative calculus and many concepts such as differentiation and integration are based on these operations. Being the main concept of this paper, convergence of sequences of positive real numbers is also defined via these operations in multiplicative calculus. In this paper, we use multiplicative calculus to deal with the convergence of sequences of real numbers through weighted geometric means. By the way, weighted geometric means are encountered in many topics of mathematics one of which is sequences of IFNs. In particular, see intuitionistic fuzzy aggregation operators\cite{xu-yager,xu}.

There are many examples of sequences in real number space and in intuitionistic fuzzy number space where the convergence can not be achieved via existing types of convergence. Besides, in some cases, the limit may not be unique or may not be the intended value even if the convergence is achieved via those types of convergence. To recover the convergence of such sequences, we need new types-methods of convergence. The main aim of this paper is to introduce weighted geometric mean method of convergence for sequences in $\mathbb{R}^+$ by using multiplicative calculus and prove related convergence theorems in $\mathbb{R}^+$ with application to intuitionistic fuzzy number space. Recently, \c{C}anak et al.\cite{canak} used multiplicative calculus and defined geometric mean method to assign a limit value to sequences which fail to converge in $\mathbb{R}^+$. Besides, they obtained conditions in the multiplicative sense under which convergence in $\mathbb{R}^+$ follows from convergence of geometric means and gave multiplicative analogues of Schmidt type slow oscillation condition and Landau type two-sided condition as corollaries. In Example \ref{ex1}, geometric mean method fails to assign a limit. Furthermore, in Example \ref{ex2}, geometric mean method may not assign the intended limit value even if it achieves a limit value in $\mathbb{R}^+$. In such cases, we may use weighted geometric means instead of geometric means. Hence, by using multiplicative calculus, in Section \ref{section3} we define weighted geometric mean method for sequences in $\mathbb{R}^+$ and obtain necessary and sufficient conditions under which convergence in $\mathbb{R}^+$ follows from convergence of weighted geometric means. In Section \ref{section4}, we define two types of convergence and weighted mean limitation methods for sequences of IFNs to handle sequences such in Example \ref{nonunique}, Example \ref{ex3}, Example \ref{ex4} and apply the conditions of Section \ref{section3} to sequences of IFNs in order to recover the convergence.

\section{Definitions and Notations}
Let $u\in\mathbb{R}^+$. Then, absolute value of $u$ in the multiplicative sense is\cite{bashirov}
\begin{eqnarray*}
|u|^*=
\begin{cases}
u, \quad &u\geq1,\\
\frac{1}{u}, &u<1.
\end{cases}
\end{eqnarray*}
Let $u,v\in\mathbb{R}^+$. Then, properties below are valid.\cite{abbas}
\begin{itemize}
\item[i)] $|u|^*\geq1$
\item[ii)] $|\frac{1}{u}|^*=|u|^*$
\item[iii)] $|u|^*\leq v$ if and only if $\frac{1}{v}\leq u\leq v$
\item[iv)] $|uv|^*\leq |u|^*|v|^*$.
\end{itemize}
Multiplicative distance is defined by\cite{bashirov}
\begin{eqnarray*}
d^*(u,v)=\left|\frac{u}{v}\right|^*
\end{eqnarray*}
and satisfies the following properties:
\begin{itemize}
\item[i)] $d^*(u,v)\geq1$ for all $u,v\in\mathbb{R}^+$,
\item[ii)] $d^*(u,v)=1$ if and only if $u=v$,
\item[iii)] $d^*(u,v)=d^*(v,u)$ for all $u,v\in\mathbb{R}^+$,
\item[iv)] $d^*(u,v)\leq d^*(u,z)d^*(z,v)$ for all $u,z,v\in\mathbb{R}^+$.
\end{itemize}
A sequence $(u_n)$ in $(\mathbb{R}^+, |\cdot|^*)$ is said to be *convergent to $a\in \mathbb{R}^+$ if for all $\varepsilon>1$ there exists $n_0\in\mathbb{N}$ such that $d^*(u_n,a)=\left|\frac{u_n}{a}\right|^*<\varepsilon$ whenever $n>n_0$, and denoted by $u_n\stackrel{*}{\to}a$. Sequence $(u_n)$ is said to be *bounded if there exists $B>1$ such that $\left|u_n\right|^*<B$ for all $n\in\mathbb{N}$. For further concepts such as multiplicative derivative, multiplicative differential equations and the Newtonian counterparts, see \cite{katz,stanley,bashirov,abbas,tor,newton1,newton2,newton3}.
\begin{remark}\label{remark1}
We note that a sequence $(u_n)$ in $\mathbb{R}^+$ *converges to $a\in \mathbb{R}^+$ if and only if $(u_n)$ converges to $a$ with respect to usual absolute value metric in $\mathbb{R}^+$. That is, *convergence and convergence are equivalent in $\mathbb{R}^+$. On the other hand, the same is not valid for *boundedness and boundedness in $\mathbb{R}^+$ which can be seen by the sequence $(u_n)=(1/n)$ that is bounded in $(\mathbb{R}^+, |\cdot|)$ but is not *bounded in $(\mathbb{R}^+, |\cdot|^*)$. See also \cite[Section 4.3]{bashirov}.
\end{remark}
Let $(p_n)$ be a sequence of nonnegative numbers such that
\begin{eqnarray}\label{p}
P_n:=\sum_{k=0}^np_k\to \infty\quad as\quad n\to\infty\qquad (p_0>0)
\end{eqnarray}
and $\lambda_n=\lfloor\lambda n\rfloor$ for a positive number $\lambda$. $SV\!A_+$ is the set of all sequences $p=(p_n)$ satisfying\cite{chen}
\begin{eqnarray*}
\liminf_{n\to\infty}\left|\frac{P_{\lambda_n}}{P_n}-1\right|>0\quad for \ each\ \lambda>0 \ with\ \lambda\neq1.
\end{eqnarray*}
\begin{definition}\cite{canak}
A sequence $(u_n)$ of positive real numbers is said to be $*-$slowly oscillating if
\begin{eqnarray*}
\liminf_{\lambda\to1^+}\limsup_{n\to\infty}\max_{n<m\leq\lambda_n}\left|\frac{u_m}{u_n}\right|^*=1,
\end{eqnarray*}
or equivalently
\begin{eqnarray*}
\liminf_{\lambda\to1^-}\limsup_{n\to\infty}\max_{\lambda_n<m\leq n}\left|\frac{u_n}{u_m}\right|^*=1.
\end{eqnarray*}
\end{definition}
Now, we give some definitions concerning intuitionistic fuzzy sets which are necessary for Section \ref{section4}.

\noindent Let $X$ be a non-empty set. Then, an Atanassov's intuitionistic fuzzy set(A-IFS)\cite{atanassov} has the following form: $A=\{\langle x,\mu_{A}(x),\nu_{A}(x)|x\in X\rangle\}$ where $\mu:X\to[0,1]$ is called membership function and $\nu:X\to[0,1]$ is called non-membership function. For any $x\in X$, $0\leq \mu_{A}(x)+\nu_{A}(x)\leq1$. In special case $\mu_{A}(x)+\nu_{A}(x)=1$, A-IFS degenerates to fuzzy set\cite{zadeh}. For convenience, Xu and Yager\cite{xu-yager} called $\alpha=(\mu_{\alpha},\nu_{\alpha})$ an IFN which satisfies $\mu_{\alpha}\in[0,1]$, $\nu_{\alpha}\in[0,1]$, and  $0\leq \mu_{\alpha}+\nu_{\alpha}\leq1$.
\begin{definition}\label{totalorder}\cite{xu-yager,xu}
Let $\alpha_1=(\mu_1,\nu_1)$ and $\alpha_2=(\mu_2,\nu_2)$ be two IFNs, $s(\alpha_1)=\mu_1-\nu_1$ and $s(\alpha_2)=\mu_2-\nu_2$ the scores of $\alpha_1$ and $\alpha_2$, respectively,
$h(\alpha_1)=\mu_1+\nu_1$ and $h(\alpha_2)=\mu_2+\nu_2$ the accuracy degrees of the $\alpha_1$ and $\alpha_2$, respectively. Then,
\begin{itemize}
\item If $s(\alpha_1)<s(\alpha_2)$, then $\alpha_1$ is smaller than $\alpha_2$, denoted by $\alpha_1<\alpha_2$.
\item If $s(\alpha_1)=s(\alpha_2)$, then
\item[(1)] If $h(\alpha_1)=h(\alpha_2)$, then $\alpha_1$ equal to $\alpha_2$, i.e., $\mu_1=\mu_2$, $\nu_1=\nu_2$ denoted by $\alpha_1=\alpha_2$;
\item[(2)] If $h(\alpha_1)<h(\alpha_2)$, then $\alpha_1$ is smaller than $\alpha_2$, denoted by $\alpha_1<\alpha_2$;
\item[(3)] If $h(\alpha_1)>h(\alpha_2)$, then $\alpha_1$ is larger than $\alpha_2$, denoted by $\alpha_1>\alpha_2$.
\end{itemize}
\end{definition}
\begin{definition}\label{partialorder}\cite{kerre}
Let $\alpha_1=(\mu_1,\nu_1)$ and $\alpha_2=(\mu_2,\nu_2)$ be two IFNs. Then
\begin{itemize}
\item[(1)] If $\mu_1>\mu_2$ and $\nu_1<\nu_2$, then $\alpha_1>_{\mathsmaller{L}}\alpha_2$
\item[(2)] If $\mu_1<\mu_2$ and $\nu_1>\nu_2$, then $\alpha_1<_{\mathsmaller{L}}\alpha_2$
\item[(3)] If $\mu_1=\mu_2$ and $\nu_1=\nu_2$, then $\alpha_1=\alpha_2$
\end{itemize}
\end{definition}
\begin{definition}\label{substraction}\cite{xu-yager,xu,type1}
Let $\alpha_1=(\mu_1,\nu_1)$, $\alpha_2=(\mu_2,\nu_2)$ be two IFNs and $c\geq0$. Then,
\begin{itemize}
\item[(1)] $\alpha_1\oplus\alpha_2=(1-(1-\mu_1)(1-\mu_2),\nu_1\nu_2)$
\item[(2)]
\begin{eqnarray*}
\displaystyle\alpha_1\ominus\alpha_2=
\begin{cases}\displaystyle\left(\frac{\mu_{1}-\mu_{2}}{1-\mu_{2}},\frac{\nu_{1}}{\nu_{2}}\right), \quad&\parbox[t]{5.2cm}{\it if\ $\mu_{1}\geq\mu_{2}$,\ $\nu_{1}\leq\nu_{2}$,\ $\nu_{2}>0$\ and\ $\nu_{1}\pi_{\alpha_2}\leq\pi_{\alpha_1}\nu_{2}$}\\[5mm]
(0,1),&otherwise
\end{cases}
\end{eqnarray*}
where $\pi_{\alpha_1}=1-\mu_1-\nu_1$ and $\pi_{\alpha_2}=1-\mu_2-\nu_2$
\item[(3)] $\alpha_1\otimes\alpha_2=(\mu_1\mu_2,1-(1-\nu_1)(1-\nu_2))$
\item[(4)] $c\alpha_1=(1-(1-\mu_1)^{c},\nu_1^{c}),\ where \ \alpha_1<_{\mathsmaller{L}}(1,0)$
\item[(5)] $\alpha_1^c=(\mu_1^{c},1-(1-\nu_1)^{c}),\ where\ \alpha_1>_{\mathsmaller{L}}(0,1)$
\end{itemize}
The addition region of IFN $\xi$ is defined by $A^{\oplus}_{\xi}=\{\alpha\mid \alpha=\xi\oplus\beta,\ \beta\in IFNS\}$. Let $(\alpha_n)$ be a sequence of IFNs where $\alpha_n=(\mu_n,\nu_n)$ . We call $(\alpha_n)$ an addition sequence of $\xi$ if there exists $n_0\in\mathbb{N}$ such that $\alpha_n\in A^{\oplus}_{\xi}$ for all $n>n_0$\cite{type1}.
\end{definition}
\section{Main Results}\label{section3}
In this section, we first define weighted geometric mean method of convergence in $\mathbb{R}^+$ by using multiplicative calculus in order to achieve the convergence of sequences such in Example \ref{ex1} and Example \ref{ex2}. Then, we give various conditions under which convergence of sequences in $\mathbb{R}^+$ follows from convergence of their weighted geometric means.

We note that $(u_n)$ denotes a sequence in $\mathbb{R}^+$ throughout this section. Let $(p_n)$ be a sequence of nonnegative real numbers satisfying \eqref{p}. The weighted geometric means $(w_n)$ of sequence $(u_n)$ is defined by
\begin{eqnarray}\label{w}
w_n=\left(\prod_{k=0}^nu_k^{p_k}\right)^{\frac{1}{P_n}}, \quad  n=0,1,2,....
\end{eqnarray}%
We say that $(u_n)$ *converges to $a\in\mathbb{R}^+$ by weighted geometric mean method, briefly: $(\bar{G},p)-$*convergent to $a$ if $w_n\stackrel{*}{\to}a$. When $p_n=1$ for all $n\in\mathbb{N}$, $(\bar{G},p)-$*convergence reduces to geometric mean method defined in \cite{canak}. For weighted arithmetic means, see \cite{moricz,chen,canak2}.
\begin{theorem}\label{regular}
If sequence $(u_n)$ is *convergent to $a\in\mathbb{R}^+$, then $(u_n)$ is $(\bar{G},p)-$*convergent to $a$.
\end{theorem}
\begin{proof}
Let $(u_n)$ be *convergent to $a\in\mathbb{R}^+$. Then, for given $\varepsilon>1$ there exist $n_0\in \mathbb{N}$ and $B>1$ such that $\left|\frac{u_n}{a}\right|^*<\varepsilon^{1/2}$ for $n>n_0$ and $\left|\frac{u_n}{a}\right|^*<B$ for $n\leq n_0$. Besides, there is $n_1\in \mathbb{N}$ such that $B^{\frac{P_{n_0}}{P_n}}<\varepsilon^{1/2}$ in view of the fact $\lim_{n\to\infty}B^{\frac{P_{n_0}}{P_n}}=1$. Hence, we have
\begin{eqnarray*}
\left|\frac{w_n}{a}\right|^*&=&\left|\left(\prod_{k=0}^nu_k^{p_k}\right)^{\frac{1}{P_n}}\Bigg/\left(\prod_{k=0}^na^{p_k}\right)^{\frac{1}{P_n}}\right|^*
\\&=&
\left|\left\{\prod_{k=0}^n\left(\frac{u_k}{a}\right)^{p_k}\right\}^{\frac{1}{P_n}}\right|^*
\\&\leq&
\left\{\prod_{k=0}^n\left(\left|\frac{u_k}{a}\right|^*\right)^{p_k}\right\}^{\frac{1}{P_n}}
\\&=&
\left\{\prod_{k=0}^{n_0}\left(\left|\frac{u_k}{a}\right|^*\right)^{p_k}\right\}^{\frac{1}{P_n}}\left\{\prod_{k=n_0+1}^{n}\left(\left|\frac{u_k}{a}\right|^*\right)^{p_k}\right\}^{\frac{1}{P_n}}
\\&<&
B^{\frac{P_{n_0}}{P_n}}(\varepsilon^{1/2})^{\frac{P_n-P_{n_0}}{P_n}}
\\&<&
B^{\frac{P_{n_0}}{P_n}}\varepsilon^{1/2}
\\&<&
\varepsilon
\end{eqnarray*}
for $n>\max\{n_0,n_1\}$ which completes the proof.
\end{proof}
The converse of Theorem \ref{regular} is not true in general. That is, $(\bar{G},p)-$*convergence does not imply *convergence in $\mathbb{R}^+$ which can be seen by the following examples.
\begin{example}\label{ex1}
Sequence $(u_n)$ defined by $u_n=e^{(-1)^{n}(n+1)}$ is not *convergent and not $G-$*convergent introduced by \cite{canak}, but it is $(\bar{G},p)-$*convergent to 1 for $p_n=\frac{1}{n+1}$.
\end{example}
\begin{example}\label{ex2}
Sequence $(u_n)$ defined by $u_n=2^{(-1)^n}$ is not *convergent, but it is $(\bar{G},p)-$*convergent to 1 for $p_n=1$ and $(\bar{G},p)-$*convergent to $\sqrt[3]{2}$ for
\begin{eqnarray*}
p_n=\begin{cases}
1, \quad &n\ is\ odd,\\
2, &n\ is\ even.
\end{cases}
\end{eqnarray*}
With appropriate choice of the weights $p_n$, the other intended values in $[1/2,2]$ can also be assigned as limit value to sequence $(u_n)$ by the help of weighted geometric mean method.
\end{example}
In this section we give some conditions for $(\bar{G},p)-$*convergence to imply *convergence in $\mathbb{R}^+$. We need following lemma to prove main results of this section.
\begin{lemma}
$(i)$ Let $\lambda>1$. For each $n$ such that $P_{\lambda_n}>P_n$ we have
\begin{eqnarray}\label{eq1}
\frac{u_n}{w_n}=\left(\frac{w_{\lambda_n}}{w_n}\right)^{\frac{P_{\lambda_n}}{P_{\lambda_n}-P_n}}\left[\left\{\prod_{k=n+1}^{\lambda_n}\left(\frac{u_k}{u_n}\right)^{p_k}\right\}^{\frac{1}{P_{\lambda_n}-P_n}}\right]^{-1}.
\end{eqnarray}
$(ii)$ Let $0<\lambda<1$. For each $n$ such that $P_n>P_{\lambda_n}$ we have
\begin{eqnarray}\label{eq2}
\frac{u_n}{w_n}=\left(\frac{w_n}{w_{\lambda_n}}\right)^{\frac{P_{\lambda_n}}{P_n-P_{\lambda_n}}}\left\{\prod_{k=\lambda_n+1}^{n}\left(\frac{u_n}{u_k}\right)^{p_k}\right\}^{\frac{1}{P_n-P_{\lambda_n}}}.
\end{eqnarray}
\end{lemma}
\begin{proof}
$(i)$ By the fact that
\begin{eqnarray*}
\frac{w_{\lambda_n}}{w_n}&=&\frac{\left(\prod_{k=0}^{\lambda_n}u_k^{p_k}\right)^{\frac{1}{P_{\lambda_n}}}}{\left(\prod_{k=0}^{n}u_k^{p_k}\right)^{\frac{1}{P_{n}}}}
\\&=&
\left(\prod_{k=0}^{n}u_k^{p_k}\right)^{\frac{P_n-P_{\lambda_n}}{P_nP_{\lambda_n}}}\left(\prod_{k=n+1}^{\lambda_n}u_k^{p_k}\right)^{\frac{1}{P_{\lambda_n}}}
\\&=&
w_n^{\frac{P_n-P_{\lambda_n}}{P_{\lambda_n}}}\left\{\prod_{k=n+1}^{\lambda_n}\left(\frac{u_k}{u_n}\right)^{p_k}u_n^{p_k}\right\}^{\frac{1}{P_{\lambda_n}}}
\\&=&
w_n^{\frac{P_n-P_{\lambda_n}}{P_{\lambda_n}}}\left\{\prod_{k=n+1}^{\lambda_n}\left(\frac{u_k}{u_n}\right)^{p_k}\right\}^{\frac{1}{P_{\lambda_n}}}u_n^{\frac{P_{\lambda_n}-P_n}{P_{\lambda_n}}},
\end{eqnarray*}
we have
\begin{eqnarray*}
\left(\frac{w_{\lambda_n}}{w_n}\right)^{\frac{P_{\lambda_n}}{P_{\lambda_n}-P_n}}&=&\frac{u_n}{w_n}\left\{\prod_{k=n+1}^{\lambda_n}\left(\frac{u_k}{u_n}\right)^{p_k}\right\}^{\frac{1}{P_{\lambda_n}-P_n}}
\end{eqnarray*}
which implies \eqref{eq1}.

$(ii)$ The proof of \eqref{eq2} can be done similarly, hence it is omitted.
\end{proof}
We now give the necessary and sufficient conditions under which $(\bar{G},p)-$*convergence implies *convergence in $\mathbb{R}^+$.
\begin{theorem}\label{main}
Let $(p_n)\in SV\!A_+$. If $(u_n)$ is $(\bar{G},p)-$*convergent to  $a\in\mathbb{R}^+$, then $(u_n)$ is *convergent to $a$ if and only if one of the following two conditions hold:
\begin{eqnarray}\label{con1}
\liminf_{\ \lambda\rightarrow 1^+}\limsup_{n\rightarrow \infty}\left\{\left|\prod_{k=n+1}^{\lambda_n}\left(\frac{u_k}{u_n}\right)^{p_k}\right|^*\right\}^{\frac{1}{P_{\lambda_n}-P_n}}=1
\end{eqnarray}
or
\begin{eqnarray}\label{con2}
\liminf_{\ \lambda\rightarrow 1^-}\limsup_{n\rightarrow \infty}\left\{\left|\prod_{k=\lambda_n+1}^{n}\left(\frac{u_n}{u_k}\right)^{p_k}\right|^*\right\}^{\frac{1}{P_n-P_{\lambda_n}}}=1.
\end{eqnarray}
\end{theorem}
\begin{proof}
{\it Necessity.} Suppose $u_n\stackrel{*}{\to}a$. Let $\lambda>1$. Then, from Theorem \ref{regular} we have $w_n\stackrel{*}{\to}a$ which implies
\begin{eqnarray}\label{con3}
\lim_{n\to\infty}\left|\frac{u_n}{w_n}\right|^*=1,
\end{eqnarray}
and
\begin{eqnarray}\label{con4}
\lim_{n\to\infty}\left(\left|\frac{w_{\lambda_n}}{w_n}\right|^*\right)^{\frac{P_{\lambda_n}}{P_{\lambda_n}-P_n}}=1,
\end{eqnarray}
in view of $(p_n)\in SVA_+$. Also, from equation \eqref{eq1} we have
\begin{eqnarray}\label{con5}
\left(\left|\prod_{k=n+1}^{\lambda_n}\left(\frac{u_k}{u_n}\right)^{p_k}\right|^*\right)^{\frac{1}{P_{\lambda_n}-P_n}}\leq\left|\frac{u_n}{w_n}\right|^*\left(\left|\frac{w_{\lambda_n}}{w_n}\right|^*\right)^{\frac{P_{\lambda_n}}{P_{\lambda_n}-P_n}}.
\end{eqnarray}
Then, \eqref{con1} is satisfied by virtue of \eqref{con3},\eqref{con4} and \eqref{con5}.

Let $0<\lambda<1$. Since $w_n\stackrel{*}{\to}a$, we have
\begin{eqnarray}\label{con6}
\lim_{n\to\infty}\left\{\left|\frac{w_n}{w_{\lambda_n}}\right|^*\right\}^{\frac{P_{\lambda_n}}{P_n-P_{\lambda_n}}}=1
\end{eqnarray}
in view of $(p_n)\in SVA_+$. Also, from equation \eqref{eq2} we have
\begin{eqnarray}\label{con7}
\left\{\left|\prod_{k=\lambda_n+1}^{n}\left(\frac{u_n}{u_k}\right)^{p_k}\right|^*\right\}^{\frac{1}{P_n-P_{\lambda_n}}}\leq\left|\frac{u_n}{w_n}\right|^*\left\{\left|\frac{w_n}{w_{\lambda_n}}\right|^*\right\}^{\frac{P_{\lambda_n}}{P_n-P_{\lambda_n}}}.
\end{eqnarray}
Then, \eqref{con2} is satisfied by virtue of \eqref{con3},\eqref{con6} and \eqref{con7}.

{\it Sufficiency.} Suppose $(u_n)$ is $(\bar{G},p)-$*convergent to $a$ and \eqref{con1} is satisfied. Then from \eqref{con1} there exists $\lambda_j\downarrow 1$ such that
\begin{eqnarray}\label{con8}
\lim_{j\to\infty}\limsup_{n\to \infty}\left\{\prod_{k=n+1}^{\lambda_{jn}}\left(\frac{u_k}{u_n}\right)^{p_k}\right\}^{\frac{1}{P_{\lambda_{jn}}-P_n}}=1
\end{eqnarray}
where $\lambda_{jn}=\lfloor \lambda_jn\rfloor$. Also by equality \eqref{eq1} we have
\begin{eqnarray}\label{con9}
\limsup_{n\to\infty}\left|\frac{u_n}{w_n}\right|^*=\lim_{j\to\infty}\limsup_{n\to \infty}\left\{\left|\frac{w_{\lambda_n}}{w_n}\right|^*\right\}^{\frac{P_{\lambda_{jn}}}{P_{\lambda_{jn}}-P_n}}\lim_{j\to\infty}\limsup_{n\to \infty}\left\{\left|\prod_{k=n+1}^{\lambda_{jn}}\left(\frac{u_k}{u_n}\right)^{p_k}\right|^*\right\}^{\frac{1}{P_{\lambda_{jn}}-P_n}}.
\end{eqnarray}
Hence, from \eqref{con4}, \eqref{con8} and  \eqref{con9} we get
\begin{eqnarray*}
\limsup_{n\to\infty}\left|\frac{u_n}{w_n}\right|^*=1,
\end{eqnarray*}
which implies $u_n\stackrel{*}{\to}a$ in view of the fact that $w_n\stackrel{*}{\to}a$.

Suppose $(u_n)$ is $(\bar{G},p)-$*convergent to $a$ and \eqref{con2} is satisfied. Then, from \eqref{con2} there exists $\lambda_j\uparrow 1$ such that
\begin{eqnarray}\label{con10}
\lim_{j\to\infty}\limsup_{n\to \infty}\left\{\left|\prod_{k=\lambda_{jn}+1}^{n}\left(\frac{u_n}{u_k}\right)^{p_k}\right|^*\right\}^{\frac{1}{P_n-P_{\lambda_{jn}}}}=1.
\end{eqnarray}
Also by equality \eqref{eq2} we have
\begin{eqnarray}\label{con11}
\limsup_{n\to\infty}\left|\frac{u_n}{w_n}\right|^*=\lim_{j\to\infty}\limsup_{n\to \infty}\left(\left|\frac{w_n}{w_{\lambda_{jn}}}\right|^*\right)^{\frac{P_{\lambda_{jn}}}{P_n-P_{\lambda_{jn}}}}\lim_{j\to\infty}\limsup_{n\to \infty}\left\{\left|\prod_{k=\lambda_{jn}+1}^{n}\left(\frac{u_n}{u_k}\right)^{p_k}\right|^*\right\}^{\frac{1}{P_n-P_{\lambda_{jn}}}}.
\end{eqnarray}
Hence, from \eqref{con6}, \eqref{con10} and  \eqref{con11} we get
\begin{eqnarray*}
\limsup_{n\to\infty}\left|\frac{u_n}{w_n}\right|^*=1,
\end{eqnarray*}
which implies $u_n\stackrel{*}{\to}a$ in view of the fact that $w_n\stackrel{*}{\to}a$.
\end{proof}
In view of Theorem \ref{main}, we give following corollary as result of the fact that $*-$slowly oscillation implies \eqref{con1} and \eqref{con2}.
\begin{corollary}\label{slowly}
If $(u_n)$ is $(\bar{G},p)-$*convergent to $a\in\mathbb{R}^+$ and $*-$slowly oscillating, then it is *convergent to $a$.
\end{corollary}
\begin{lemma}\label{one-sided}
If $\left(\left(\Delta^*u_n\right)^n\right)$ is *bounded then $(u_n)$ is $*-$slowly oscillating, where
\begin{eqnarray*}
\Delta^*u_n=\frac{u_n}{u_{n-1}} \ \textrm{for} \ n\geq1\ \text{and}\ \Delta^*u_0=u_0.
\end{eqnarray*}
\end{lemma}
\begin{proof}
Let $\left(\left(\Delta^*u_n\right)^n\right)$ be *bounded. Then, there is $H>1$ such that $\left|\left(\Delta^*u_n\right)^n\right|^*<H$ for every $n\in\mathbb{N}$. Let $\varepsilon>1$ be given. Then, for $1<\lambda<1+\log_{H}{\varepsilon}$ and $n<m\leq\lambda_n$ we get
\begin{eqnarray*}
\left|\frac{u_m}{u_n}\right|^*=\left|\prod_{k=n+1}^{m}\!\!\!\Delta^*u_k\right|^*\leq \prod_{k=n+1}^{m}\!\!\!\left|\Delta^*u_k\right|^*<\prod_{k=n+1}^{m} H^{1/k}<H^{\frac{m-n}{n}}\leq H^{\lambda-1}<\varepsilon,
\end{eqnarray*}
which implies that $(u_n)$ is $*-$slowly oscillating.
\end{proof}
In view of Corollary \ref{slowly} and Lemma \ref{one-sided} we give following results.
\begin{corollary}\label{cor1}
Let $(p_n)\in SV\!A_+$. If $(u_n)$ is $(\bar{G},p)-$*convergent to $a\in\mathbb{R}^+$ and $\left(\left(\Delta^*u_n\right)^n\right)$ is *bounded, then $(u_n)$ is *convergent to $a$.
\end{corollary}
\begin{corollary}\label{cor2}
Let $(p_n)\in SV\!A_+$. If $(u_n)$ is $(\bar{G},p)-$*convergent to $a\in\mathbb{R}^+$ and $\left(\Delta^*u_n\right)^n\stackrel{*}{\to}1$, then $(u_n)$ is *convergent to $a$.
\end{corollary}
\begin{remark}
In view of Remark \ref{remark1}, "*convergent" can be replaced by "convergent" in the theorems and corollaries of this section.
\end{remark}
\section{Convergence of sequences of IFNs}\label{section4}
Authors have done many studies concerning intuitionistic fuzzy sets\cite{type1,mursaleen1,mursaleen2,mursaleen3,type2,type3,type4,type5,zhang,alcantud1,alcantud2,alcantud3,type6}. Among them, Lei and Xu\cite{type1} are the first to define convergence of sequences of IFNs by using subtraction operation.
\begin{definition}\label{partialconvergence}\cite{type1}
Let $(\alpha_n)$ be an addition sequence of IFN $\xi$. If $\forall\bar{\varepsilon}>_{\mathsmaller{L}}(0,1)$ there is a positive integer $n_0$ such that
\begin{eqnarray*}
\alpha_n\ominus\xi<_{\mathsmaller{L}}\bar{\varepsilon}
\end{eqnarray*}
for $n>n_0$, then $\xi$ is the addition limit of $(\alpha_n)$.
\end{definition}
In this definition, Lei and Xu\cite{type1} implemented the assumption that $(\alpha_n)$ is an addition sequence of $\xi$ in order to guarantee $\alpha_n\ominus\xi$ to be an IFN, by which we mean $\left(\frac{\mu_{n}-\mu_{\xi}}{1-\mu_{\xi}},\frac{\nu_{n}}{\nu_{\xi}}\right)$ is an IFN. Besides, they proved that this convergence, under the same assumption, satisfies Theorem \ref{levels+} and Theorem \ref{levelsx} which are very useful in calculation of limits of sequences of IFNs. However, there are many sequences of IFNs which are not addition sequences of the limit points as in Example \ref{nonunique} and hence Definition \ref{partialconvergence} is not applicable to such sequences. Zhang and Xu\cite{zhang} removed the assumption on the sequence and defined following convergence by using the addition operation and the order relation given in Definition \ref{totalorder}.
\begin{definition}\label{totalconvergence}\cite{zhang}
Let $(\alpha_n)$ be a sequence of IFNs and $\xi$ be an IFN. Sequence $(\alpha_n)$ is said to be convergent  to $\xi$ if for any IFN $\bar{\varepsilon}$, there exists $n_0\in\mathbb{N}$ such that
\begin{eqnarray*}
\begin{cases}
\alpha_n<\xi\oplus\bar{\varepsilon} & \alpha_n>\xi\\
\xi<\alpha_n\oplus\bar{\varepsilon} & \alpha_n<\xi
\end{cases}
\end{eqnarray*}
hold for $n>n_0$.
\end{definition}
In Definition \ref{totalconvergence}, there is no assumption on the sequence but the limit is not unique. Besides, this type of convergence does not satisfy Theorem \ref{levels+} and Theorem \ref{levelsx} which are useful theorems. These can be seen in the following example.
\begin{example}\label{nonunique}
Consider the sequence of IFNs defined by $\alpha_n=(\mu_n,\nu_n)=\left(\frac{1}{2}-\frac{1}{n+3},\frac{1}{3}-\frac{1}{n+3}\right)$.

{\it \underline{\bf Case1}(convergence by Definition \ref{partialconvergence}).} The only candidate for the limit value is $\xi_1=\left(\frac{1}{2},\frac{1}{3}\right)$, but $(\alpha_n)$ is not addition sequence of $\xi_1$ since $\left(\frac{\mu_{n}-\mu_{\xi_1}}{1-\mu_{\xi_1}},\frac{\nu_{n}}{\nu_{\xi_1}}\right)=\left(-\frac{2}{n+3},\frac{n}{n+3}\right)$ is not an IFN for any $n\in\mathbb{N}$.  Hence, Definition \ref{partialconvergence} is not applicable.

{\it \underline{\bf Case2}(convergence by Definition \ref{totalconvergence}).}

$(\alpha_n)$ converges to IFNs $\xi_1=\left(\frac{1}{2},\frac{1}{3}\right)$ and $\xi_2=\left(\frac{7}{12},\frac{5}{12}\right)$ in view of the facts that
\begin{eqnarray*}
\alpha_n<\xi_{1}<\alpha_n\oplus\bar{\varepsilon}\quad and\quad \alpha_n<\xi_{2}<\alpha_n\oplus\bar{\varepsilon}
\end{eqnarray*}
for any IFN $\bar{\varepsilon}\neq(0,1)$ and $n\in\mathbb{N}$. In fact, $(\alpha_n)$ has infinite number of limits since $(\alpha_n)$ converges to any IFN $\xi$ such that
\begin{eqnarray*}
\left\{\xi\mid \mu_{\xi}-\nu_{\xi}=\frac{1}{6}\ \text{and} \ \mu_{\xi}+\nu_{\xi}\geq\frac{5}{6}\right\}\cdot
\end{eqnarray*}
Hence, the limit is not unique. On the other hand, Theorem \ref{levels+} and Theorem \ref{levelsx} are not satisfied since $\lim_{n\to\infty}\alpha_n=\xi_2$ but $\lim_{n\to\infty}\mu_n\neq7/12$ and $\lim_{n\to\infty}\nu_n\neq5/12$.
\end{example}
In this section, following \cite{type1} and \cite{zhang}, we first define the concepts of $^{\oplus}$convergence and $^{\otimes}$convergence for sequences of IFNs by means of the partial order given in Definition \ref{partialorder}. Then, we apply the results of Section \ref{section3} in order to achieve convergence in intuitionistic fuzzy number space.
\begin{definition}
Let $(\alpha_n)$ be a sequence of IFNs and $\xi$ be an IFN. Sequence $(\alpha_n)$ is said to be $^{\oplus}$convergent  to $\xi$ if for any IFN $\bar{\varepsilon}=(\varepsilon,1-\varepsilon)>_{\mathsmaller{L}}(0,1)$, there exists $n_0\in\mathbb{N}$ such that
\begin{eqnarray*}
\alpha_n<_{\mathsmaller{L}}\xi\oplus\bar{\varepsilon}\quad and \quad\xi<_{\mathsmaller{L}}\alpha_n\oplus\bar{\varepsilon}
\end{eqnarray*}
hold for $n>n_0$.
\end{definition}
\begin{theorem}\label{levels+}
A sequence $(\alpha_n)$ of IFNs $^{\oplus}$converges  to an IFN $\xi\!\!<_{\mathsmaller{L}}\!\!\!(1,0)$ if and only if $\lim_{n\to\infty}\mu_n=\mu_{\xi}$ and $\lim_{n\to\infty}\nu_n=\nu_{\xi}$.
\end{theorem}
\begin{proof}
{\it Necessity.} Suppose $(\alpha_n)$ $^{\oplus}$converges to $\xi$. Let $\varepsilon>0$ be given. Then, for $n>n_0(\varepsilon)$ we have
\begin{eqnarray*}
\mu_{n}<1-(1-\mu_{\xi})(1-\varepsilon)\quad and \quad  \mu_{\xi}<1-(1-\mu_{n})(1-\varepsilon)
\end{eqnarray*}
and
\begin{eqnarray*}
\nu_{\xi}(1-\varepsilon)<\nu_{n}\quad and \quad \nu_n(1-\varepsilon)<\nu_{\xi}.
\end{eqnarray*}
Since $\varepsilon>0$ is arbitrary, this implies $\lim_{n\to\infty}\mu_n=\mu_{\xi}$ and $\lim_{n\to\infty}\nu_n=\nu_{\xi}$.

{\it Sufficiency.} Let $\lim_{n\to\infty}\mu_n=\mu_{\xi}$ and $\lim_{n\to\infty}\nu_n=\nu_{\xi}$. For given $\varepsilon>0$ followings hold:

$(i)$ There exists $n_1\in\mathbb{N}$ such that $\mu_n-\mu_{\xi}<\varepsilon(1-\mu_{\xi})$ and $\nu_{\xi}-\nu_n<\varepsilon\nu_{\xi}$ for $n>n_1$ and these imply $\mu_{n}<1-(1-\mu_{\xi})(1-\varepsilon)$ and $\nu_{\xi}(1-\varepsilon)<\nu_{n}$, respectively. Hence, we have $\alpha_n<_{\mathsmaller{L}}\xi\oplus\bar{\varepsilon}$ for $n>n_1$.

$(ii)$ By the assumption $\xi<_{\mathsmaller{L}}(1,0)$ we have $\mu_{\xi}\neq1$ and $\nu_{\xi}\neq0$ and so there exists $n_2\in\mathbb{N}$ such that $\mu_n<\mu_{\xi}+\frac{1-\mu_{\xi}}{2}=\frac{\mu_{\xi}+1}{2}$ and $\nu_n>\nu_{\xi}-\frac{\nu_{\xi}}{2}=\frac{\nu_{\xi}}{2}$ for $n>n_2$. On the other hand, there exists $n_3\in\mathbb{N}$ such that $\mu_{\xi}-\mu_n<\varepsilon(1-\frac{\mu_{\xi}+1}{2})$ and $\nu_n-\nu_{\xi}<\varepsilon\frac{\nu_{\xi}}{2}$ for $n>n_3$. These imply $\mu_{\xi}-\mu_n<\varepsilon(1-\mu_n)$ and $\nu_n-\nu_{\xi}<\varepsilon\nu_n$ for $n>\max\{n_2,n_3\}$. Hence, for $n>\max\{n_2,n_3\}$ we have $\mu_{\xi}<1-(1-\mu_{n})(1-\varepsilon)$ and $\nu_{n}(1-\varepsilon)<\nu_{\xi}$ which implies $\xi<_{\mathsmaller{L}}\alpha_n\oplus\bar{\varepsilon}$.

From $(i)$ and $(ii)$, we conclude that
\begin{eqnarray*}
\alpha_n<_{\mathsmaller{L}}\xi\oplus\bar{\varepsilon}\quad and \quad\xi<_{\mathsmaller{L}}\alpha_n\oplus\bar{\varepsilon}
\end{eqnarray*}
for $n>n_0=\max\{n_1,n_2,n_3\}$ which completes the proof.
\end{proof}
Now we give $^\otimes$convergence for sequences of IFNs.
\begin{definition}
Let $(\alpha_n)$ be a sequence of IFNs and $\xi$ be an IFN. Sequence $(\alpha_n)$ is said to be $^\otimes$convergent to $\xi$ if for any IFN $\bar{\varepsilon}=(1-\varepsilon,\varepsilon)<_{\mathsmaller{L}}(1,0)$, there exists $n_0\in\mathbb{N}$ such that
\begin{eqnarray*}
\alpha_n\otimes\bar{\varepsilon}<_{\mathsmaller{L}}\xi\quad and \quad\xi\otimes\bar{\varepsilon}<_{\mathsmaller{L}}\alpha_n
\end{eqnarray*}
hold for $n>n_0$.
\end{definition}
\begin{theorem}\label{levelsx}
A sequence $(\alpha_n)$ of IFNs $^\otimes$converges to an IFN $\xi\!\!>_{\mathsmaller{L}}\!\!\!(0,1)$ if and only if $\lim_{n\to\infty}\mu_n=\mu_{\xi}$ and $\lim_{n\to\infty}\nu_n=\nu_{\xi}$.
\end{theorem}
\begin{proof}
{\it Necessity.} Suppose $(\alpha_n)$ $^\otimes$converges to $\xi$. Let $\varepsilon>0$ be given. Then, for $n>n_0(\varepsilon)$ we have
\begin{eqnarray*}
\mu_{n}(1-\varepsilon)<\mu_{\xi}\quad and \quad  \mu_{\xi}(1-\varepsilon)<\mu_{n}
\end{eqnarray*}
and
\begin{eqnarray*}
\nu_{\xi}<1-(1-\nu_n)(1-\varepsilon)\quad and \quad \nu_{n}<1-(1-\nu_{\xi})(1-\varepsilon).
\end{eqnarray*}
Since $\varepsilon>0$ is arbitrary, this implies $\lim_{n\to\infty}\mu_n=\mu_{\xi}$ and $\lim_{n\to\infty}\nu_n=\nu_{\xi}$.

{\it Sufficiency.} The proof can be done similar to sufficiency part of the proof of Theorem 16 by changing the roles of $\mu$ and $\nu$, and replacing the operation $\oplus$ by the operation $\otimes$.
\end{proof}
\begin{remark}
If the limit exists by $^\oplus$convergence, then it is unique by Theorem \ref{levels+}. Similar case is also valid for $^\otimes$convergence by Theorem \ref{levelsx}. As an example, $^{\oplus}$convergence and $^{\otimes}$convergence work in Example \ref{nonunique} with unique limit $\xi_1$.
\end{remark}
\begin{remark}
We note that if $(0,1)<_{\mathsmaller{L}}\xi<_{\mathsmaller{L}}(1,0)$, then $^{\oplus}$convergence and $^{\otimes}$convergence are equivalent in intuitionistic fuzzy number space.
\end{remark}
\subsection{Convergence via weighted arithmetic means}
In some cases of sequences of IFNs as in Example \ref{ex3}, $^{\oplus}$convergence may fail in intuitionistic fuzzy number space. In such cases we may use weighted arithmetic means to grasp a limit. In this subsection we assume $\beta<_{\mathsmaller{L}}(1,0)$ for any IFN $\beta$.
\begin{definition}
Let $(\alpha_n)$ be a sequence of IFNs and sequence $(p_n)$ of nonnegative real numbers satisfying \eqref{p}. Then, sequence of weighted arithmetic means of $(\alpha_n)$ is defined by
\begin{eqnarray*}
t_n=\frac{1}{P_n}\bigoplus_{k=0}^np_k\alpha_k\qquad\qquad \qquad\qquad(n=0,1,2,....)
\end{eqnarray*}
Sequence $(\alpha_n)$ is said to be $^{\oplus}$convergent by weighted arithmetic mean method, or $(\bar{N},p)-^{\oplus}$convergent, to IFN $\xi$ if $(t_n)$ $^{\oplus}$converges to $\xi$.
\end{definition}
We note that $(t_n)$ is, in fact, an infinite sequence of intuitionistic fuzzy weighted averaging operators (IFWA) defined by  Xu\cite{xu}.
\begin{theorem}\label{w-levels}
A sequence $(\alpha_n)$ of IFNs is $(\bar{N},p)-^{\oplus}$convergent to an IFN $\xi$ if and only if $\lim_{n\to\infty}w(1-\mu_n)=1-\mu_{\xi}$ and $\lim_{n\to\infty}w(\nu_n)=\nu_{\xi}$, where $w_n$ is weighted geometric mean operator in \eqref{w}.
\end{theorem}
\begin{proof}
Let $(\alpha_n)$ be a sequence of IFNs. We have
\begin{eqnarray*}
t_n=\frac{1}{P_n}\bigoplus_{k=0}^np_k\alpha_k=\left(1-\left\{\prod_{k=0}^n(1-\mu_k)^{p_k}\right\}^{1/P_n},\left\{\prod_{k=0}^n\nu_k^{p_k}\right\}^{1/P_n}\right)=\left(1-w(1-\mu_n),w(\nu_n)\right).
\end{eqnarray*}
By Theorem \ref{levels+}, sequence $(t_n)$ $^{\oplus}$converges to $\xi$ if and only if $\lim_{n\to\infty}w(1-\mu_n)=1-\mu_{\xi}$ and $\lim_{n\to\infty}w(\nu_n)=\nu_{\xi}$. Hence, the proof is completed.
\end{proof}
\begin{theorem}\label{ifn-regular}
If sequence $(\alpha_n)$ of IFNs is $^{\oplus}$convergent to an IFN $\xi$, then it is $(\bar{N},p)-^{\oplus}$convergent to $\xi$.
\end{theorem}
\begin{proof}
Let $(\alpha_n)$ $^{\oplus}$converge to $\xi$. From Theorem \ref{levels+} we have $\lim_{n\to\infty}(1-\mu_n)=1-\mu_{\xi}$ and $\lim_{n\to\infty}\nu_n=\nu_{\xi}$. Then, from Theorem \ref{regular} we have $\lim_{n\to\infty}w(1-\mu_n)=1-\mu_{\xi}$ and $\lim_{n\to\infty}w(\nu_n)=\nu_{\xi}$ which implies $(\alpha_n)$ is $(\bar{N},p)-^{\oplus}$convergent to $\xi$ in view of Theorem \ref{w-levels}.
\end{proof}
$(\bar{N},p)-^{\oplus}$convergence of a sequence of IFNs does not imply $^{\oplus}$convergence by the following example.
\begin{example}\label{ex3}
Consider sequence $(\alpha_n)$ of IFNs defined by
\begin{eqnarray*}
\alpha_n=(\mu_n,\nu_n)=\left(1-\left(\frac{1}{2}\right)^{(-1)^n+2},\left(\frac{1}{3}\right)^{(-1)^n+2}\right).
\end{eqnarray*}
$(\alpha_n)$ is not $^{\oplus}$convergent by Theorem \ref{levels+}. But, it is $(\bar{N},p)-^{\oplus}$convergent with $p_n=1$ to IFN $\xi=(3/4,1/9)$ by the facts that
\begin{eqnarray*}
\lim_{n\to\infty}w(1-\mu_n)=\lim_{n\to\infty}\left\{\prod_{k=0}^n\left(\frac{1}{2}\right)^{(-1)^k+2}\right\}^{\frac{1}{n+1}}=\frac{1}{4}=1-\mu_{\xi}
\end{eqnarray*}
and
\begin{eqnarray*}
\lim_{n\to\infty}w(\nu_n)=\lim_{n\to\infty}\left\{\prod_{k=0}^n\left(\frac{1}{3}\right)^{(-1)^k+2}\right\}^{\frac{1}{n+1}}=\frac{1}{9}=\nu_{\xi}.
\end{eqnarray*}
in view of Theorem \ref{w-levels}.
\end{example}
In view of Theorems \ref{main},\ref{levels+},\ref{w-levels} and Corollaries \ref{slowly},\ref{cor1},\ref{cor2}; we give the conditions under which $(\bar{N},p)-^{\oplus}$convergence implies $^{\oplus}$convergence in intuitionistic fuzzy number space.
\begin{theorem}
Let $(p_n)\in SV\!A_+$. If sequence $(\alpha_n)$ of IFNs is $(\bar{N},p)-^{\oplus}$convergent to an IFN $\xi$, then $(\alpha_n)$ is $^{\oplus}$convergent to $\xi$ if and only if one of the following two conditions hold:
\begin{eqnarray*}
\liminf_{\ \lambda\rightarrow 1^+}\limsup_{n\rightarrow \infty}\left\{\left|\prod_{k=n+1}^{\lambda_n}\left(\frac{1-\mu_k}{1-\mu_n}\right)^{p_k}\right|^*\right\}^{\frac{1}{P_{\lambda_n}-P_n}}=1,\ \ \liminf_{\ \lambda\rightarrow 1^+}\limsup_{n\rightarrow \infty}\left\{\left|\prod_{k=n+1}^{\lambda_n}\left(\frac{\nu_k}{\nu_n}\right)^{p_k}\right|^*\right\}^{\frac{1}{P_{\lambda_n}-P_n}}=1
\end{eqnarray*}
or
\begin{eqnarray*}
\liminf_{\ \lambda\rightarrow 1^-}\limsup_{n\rightarrow \infty}\left\{\left|\prod_{k=\lambda_n+1}^{n}\left(\frac{1-\mu_n}{1-\mu_k}\right)^{p_k}\right|^*\right\}^{\frac{1}{P_n-P_{\lambda_n}}}=1, \  \  \liminf_{\ \lambda\rightarrow 1^-}\limsup_{n\rightarrow \infty}\left\{\left|\prod_{k=\lambda_n+1}^{n}\left(\frac{\nu_n}{\nu_k}\right)^{p_k}\right|^*\right\}^{\frac{1}{P_n-P_{\lambda_n}}}=1.
\end{eqnarray*}
\end{theorem}
\begin{theorem}
Let $(p_n)\in SV\!A_+$. If sequence $(\alpha_n)$ of IFNs is $(\bar{N},p)-^{\oplus}$convergent to an IFN $\xi$ and sequences $(1-\mu_n)$, $(\nu_n)$ of real numbers are $*-$slowly oscillating, then $(\alpha_n)$ is $^{\oplus}$convergent to $\xi$.
\end{theorem}
\begin{theorem}
Let $(p_n)\in SV\!A_+$. If sequence $(\alpha_n)$ of IFNs is $(\bar{N},p)-^{\oplus}$convergent to an IFN $\xi$ and sequences $\left(\Delta^*(1-\mu_n)\right)^n$, $\left(\Delta^*\nu_n\right)^n$ of real numbers are *bounded, then $(\alpha_n)$ is $^{\oplus}$convergent to $\xi$.
\end{theorem}
\begin{theorem}
Let $(p_n)\in SV\!A_+$. If sequence $(\alpha_n)$ of IFNs is $(\bar{N},p)-^{\oplus}$convergent to an IFN $\xi$ and $\left(\Delta^*(1-\mu_n)\right)^n\stackrel{*}{\to}1$, $\left(\Delta^*\nu_n\right)^n\stackrel{*}{\to}1$, then $(\alpha_n)$ is $^{\oplus}$convergent to $\xi$.
\end{theorem}
\subsection{Convergence via weighted geometric means}
In some cases of sequences of IFNs as in Example \ref{ex4}, $^{\otimes}$convergence may fail in intuitionistic fuzzy number space. In such cases we may use weighted geometric means to grasp a limit. In this subsection we assume $\beta>_{\mathsmaller{L}}(0,1)$ for any IFN $\beta$.
\begin{definition}
Let $(\alpha_n)$ be a sequence of IFNs and sequence $(p_n)$ of nonnegative real numbers satisfying \eqref{p}. Then, sequence of weighted geometric means of $(\alpha_n)$ is defined by
\begin{eqnarray*}
h_n=\left\{\bigotimes_{k=0}^n\alpha_k^{p_k}\right\}^{1/P_n}\qquad\qquad \qquad\qquad(n=0,1,2,....).
\end{eqnarray*}
Sequence $(\alpha_n)$ is said to be $^{\otimes}$convergent by weighted geometric mean method, shortly $(\bar{G},p)-^{\otimes}$convergent, to IFN $\xi$ if $(h_n)$ $^{\otimes}$converges to $\xi$.
\end{definition}
We note that $(h_n)$ is, in fact, an infinite sequence of intuitionistic fuzzy weighted geometric operators (IFWG) defined by  Xu and Yager\cite{xu-yager}.
\begin{theorem}\label{w-levels2}
A sequence $(\alpha_n)$ of IFNs is $(\bar{G},p)-^{\otimes}$convergent to an IFN $\xi$ if and only if $\lim_{n\to\infty}w(\mu_n)=\mu_{\xi}$ and $\lim_{n\to\infty}w(1-\nu_n)=1-\nu_{\xi}$, where $w_n$ is weighted geometric mean operator in \eqref{w}.
\end{theorem}
\begin{proof}
Let $(\alpha_n)$ be a sequence of IFNs. We have
\begin{eqnarray*}
h_n=\left\{\bigotimes_{k=0}^n\alpha_k^{p_k}\right\}^{1/P_n}=\left(\left\{\prod_{k=0}^n\mu_k^{p_k}\right\}^{1/P_n},1-\left\{\prod_{k=0}^n(1-\nu_k)^{p_k}\right\}^{1/P_n}\right)=\left(w(\mu_n),1-w(1-\nu_n)\right).
\end{eqnarray*}
By Theorem \ref{levelsx}, sequence $(h_n)$ $^{\otimes}$converges to $\xi$ if and only if $\lim_{n\to\infty}w(\mu_n)=\mu_{\xi}$ and $\lim_{n\to\infty}w(1-\nu_n)=1-\nu_{\xi}$. Hence, the proof is completed.
\end{proof}
\begin{theorem}\label{ifn-regular2}
If sequence $(\alpha_n)$ of IFNs is $^{\otimes}$convergent to an IFN $\xi$, then it is $(\bar{G},p)-^{\otimes}$convergent to $\xi$.
\end{theorem}
\begin{proof}
Let $(\alpha_n)$ $^{\otimes}$converge to $\xi$. From Theorem \ref{levelsx} we have $\lim_{n\to\infty}\mu_n=\mu_{\xi}$ and $\lim_{n\to\infty}1-\nu_n=1-\nu_{\xi}$. Then, from Theorem \ref{regular} we have $\lim_{n\to\infty}w(\mu_n)=\mu_{\xi}$ and $\lim_{n\to\infty}w(1-\nu_n)=1-\nu_{\xi}$ which implies $(\alpha_n)$ is $(\bar{G},p)-^{\otimes}$convergent to $\xi$ in view of Theorem \ref{w-levels2}.
\end{proof}
$(\bar{G},p)-^{\otimes}$convergence of a sequence of IFNs does not imply $^{\otimes}$convergence by the following example.
\begin{example}\label{ex4}
Consider sequence $(\alpha_n)$ of IFNs defined by
\begin{eqnarray*}
\alpha_n=(\mu_n,\nu_n)=\left(\left(\frac{1}{9}\right)^{(-1)^n+2},1-\left(\frac{1}{4}\right)^{(-1)^n+2}\right).
\end{eqnarray*}
$(\alpha_n)$ is not $^{\otimes}$convergent by Theorem \ref{levelsx}. But, it is $(\bar{G},p)-^{\otimes}$convergent to IFN $\xi=(1/27,7/8)$ for
\begin{eqnarray*}
p_n=\begin{cases}
3, \quad &n\ is\ odd\\
1, &n\ is\ even
\end{cases}
\end{eqnarray*}
in view of the facts that
\begin{eqnarray*}
\lim_{n\to\infty}w(\mu_n)=\lim_{n\to\infty}\left\{\prod_{k=0}^n\left(\frac{1}{9}\right)^{p_k((-1)^k+2)}\right\}^{1/P_n}=\left(\frac{1}{9}\right)^{3/2}=\frac{1}{27}=\mu_{\xi}
\end{eqnarray*}
and
\begin{eqnarray*}
\lim_{n\to\infty}w(1-\nu_n)=\lim_{n\to\infty}\left\{\prod_{k=0}^n\left(\frac{1}{4}\right)^{p_k((-1)^k+2)}\right\}^{1/P_n}=\left(\frac{1}{4}\right)^{3/2}=\frac{1}{8}=1-\nu_{\xi}.
\end{eqnarray*}
by virtue of Theorem \ref{w-levels2}.
\end{example}
In view of Theorems \ref{main},\ref{levelsx},\ref{w-levels2} and Corollaries \ref{slowly},\ref{cor1},\ref{cor2}; we give the conditions under which $(\bar{G},p)-^{\otimes}$convergence implies $^{\otimes}$convergence in intuitionistic fuzzy number space.
\begin{theorem}\label{main2}
Let $(p_n)\in SV\!A_+$. If sequence $(\alpha_n)$ of IFNs is $(\bar{G},p)-^{\otimes}$convergent to an IFN $\xi$, then $(\alpha_n)$ is $^{\otimes}$convergent to $\xi$ if and only if one of the following two conditions hold:
\begin{eqnarray*}
\liminf_{\ \lambda\rightarrow 1^+}\limsup_{n\rightarrow \infty}\left\{\left|\prod_{k=n+1}^{\lambda_n}\left(\frac{\mu_k}{\mu_n}\right)^{p_k}\right|^*\right\}^{\frac{1}{P_{\lambda_n}-P_n}}=1,\ \ \liminf_{\ \lambda\rightarrow 1^+}\limsup_{n\rightarrow \infty}\left\{\left|\prod_{k=n+1}^{\lambda_n}\left(\frac{1-\nu_k}{1-\nu_n}\right)^{p_k}\right|^*\right\}^{\frac{1}{P_{\lambda_n}-P_n}}=1
\end{eqnarray*}
or
\begin{eqnarray*}
\liminf_{\ \lambda\rightarrow 1^-}\limsup_{n\rightarrow \infty}\left\{\left|\prod_{k=\lambda_n+1}^{n}\left(\frac{\mu_n}{\mu_k}\right)^{p_k}\right|^*\right\}^{\frac{1}{P_n-P_{\lambda_n}}}=1, \  \  \liminf_{\ \lambda\rightarrow 1^-}\limsup_{n\rightarrow \infty}\left\{\left|\prod_{k=\lambda_n+1}^{n}\left(\frac{1-\nu_n}{1-\nu_k}\right)^{p_k}\right|^*\right\}^{\frac{1}{P_n-P_{\lambda_n}}}=1.
\end{eqnarray*}
\end{theorem}
\begin{theorem}
Let $(p_n)\in SV\!A_+$. If sequence $(\alpha_n)$ of IFNs is $(\bar{G},p)-^{\otimes}$convergent to an IFN $\xi$ and sequences $(\mu_n)$, $(1-\nu_n)$ of real numbers are $*-$slowly oscillating, then $(\alpha_n)$ is $^{\otimes}$convergent to $\xi$.
\end{theorem}
\begin{theorem}
Let $(p_n)\in SV\!A_+$. If sequence $(\alpha_n)$ of IFNs is $(\bar{G},p)-^{\otimes}$convergent to an IFN $\xi$ and sequences $\left(\Delta^*\mu_n\right)^n$, $\left(\Delta^*(1-\nu_n)\right)^n$ of real numbers are *bounded, then $(\alpha_n)$ is $^{\otimes}$convergent to $\xi$.
\end{theorem}
\begin{theorem}
Let $(p_n)\in SV\!A_+$. If sequence $(\alpha_n)$ of IFNs is $(\bar{G},p)-^{\otimes}$convergent to an IFN $\xi$ and $\left(\Delta^*\mu_n\right)^n\stackrel{*}{\to}1$, $\left(\Delta^*(1-\nu_n)\right)^n\stackrel{*}{\to}1$, then $(\alpha_n)$ is $^{\otimes}$convergent to $\xi$.
\begin{remark}
The results of Subsections 4.1--4.2 can be extended to other convergence types of sequences of IFNs, provided that chosen type of convergence satisfies Theorem \ref{levels+} or Theorem \ref{levelsx}.
\end{remark}
\end{theorem}
\section{Conclusion}
In this paper, we have used multiplicative calculus and introduced weighted mean method of convergence in $\mathbb{R}^+$. We have obtained various conditions under which convergence of sequences of positive real numbers follows from convergence of their weighted geometric means. Besides, we have defined some new types of convergence for sequences of IFNs and applied the obtained conditions to occurring weighted geometric averages of membership and non-membership functions in order to grasp convergence in intuitionistic fuzzy number space. Examples of sequences such that our methods of convergence work but the methods in the literature do not work have been also given to illustrate the advantage of proposed methods. In particular,
\begin{itemize}
\item sequence of real numbers in Example \ref{ex1}: ordinary convergence and the geometric mean method\cite{canak} do not work but $(\bar{G},p)-$*convergence works,
\item sequence of real numbers in Example \ref{ex2}: $(\bar{G},p)-$*convergence is able to assign an intended limit value with appropriate choice of the weights $p=(p_n)$,
\item sequence of IFNs in Example \ref{nonunique}: Definition \ref{partialconvergence} does not work, Definition \ref{totalconvergence} works with infinite number of limits but $^{\oplus}$convergence and $^{\otimes}$convergence work with a unique limit,
\item sequences of IFNs in Example \ref{ex3} and Example \ref{ex4}: the methods in the literature do not work but $(\bar{N},p)-^{\oplus}$convergence and $(\bar{G},p)-^{\otimes}$convergence work.
\end{itemize}
The results of this paper may help researchers to handle sequences of positive real numbers and sequences of IFNs. The results may also help when dealing with sequences of weighted geometric averages occurring in many fields of science as in sequences of IFNs mentioned in Section \ref{section4}. For future work, the results may be extended to different types of fuzzy sets such as Linear Diophantine fuzzy sets\cite{dia}, Pythagorean fuzzy sets\cite{pyt1,pyt2,pyt3}, etc.
\section*{Acknowledgement}
I would like to thank Assoc. Prof. \"{O}zer Talo for his valuable comments and suggestions which helped me to improve the quality of the manuscript.
\section*{Conflict of Interest}
The authors declare that there is no conflict of interest.

\end{document}